\documentclass[12pt]{article}
 \usepackage[margin=1in]{geometry} 
\usepackage{amsmath,amsthm,amssymb,amsfonts}
 \usepackage{amssymb}

\usepackage[utf8]{inputenc}

\usepackage{graphicx}    
\usepackage{rotating}
\usepackage{color}
\usepackage{epsfig}
\usepackage{amsxtra}     
\usepackage{amsthm}
\usepackage{amssymb}
\usepackage{amsfonts}
\usepackage{pgfplots}
\pgfplotsset{compat=1.12}

\usepgfplotslibrary{fillbetween}
\usetikzlibrary{patterns}
\usepackage{hyperref}
 \usepackage{csquotes}
\usepackage[utf8]{inputenc}
\usepackage[english]{babel}
\usepackage{xcolor}
\usepackage{biblatex}
\numberwithin{equation}{section}
\newtheorem{theorem}{Theorem}[section]
\newtheorem{lemma}[theorem]{Lemma}
\newtheorem{remark}{Remark}[section]

\providecommand{\keywords}[1]
{
  \small	
  \textbf{\textit{Keywords---}} #1
}
\title{Global existence of solutions in two-species chemotaxis system with two chemicals with sub-logistic sources in 2d }
\author{ Minh Le \\
 Department of Mathematics\\
 Michigan State University\\
  Michigan, MI, 48823 \\
  \texttt{leminh2@msu.edu} }
\date{\today}

\begin{document}
\maketitle
\begin{abstract}
This paper investigates the global existence and boundedness of solutions in a two-species chemotaxis system with two chemicals and sub-logistic sources. The presence of a sub-logistic source in only one cell density equation effectively prevents the occurrence of blow-up solutions, even in fully parabolic chemotaxis systems.
\end{abstract}
\section{Introduction}
We consider a model involving the interaction of two species through chemotaxis, where each species emits a signal that influences the movement of the other species. Specifically, we study the following PDE in a open bounded domain $\Omega \subset \mathbb{R}^2$
\begin{equation} \label{KS2}
    \begin{cases}
        u_t = \Delta u -\nabla \cdot (u \nabla v)+f(u)\\
        \tau v_t = \Delta v -v+w,\\
        w_t =\Delta w -\nabla \cdot (w \nabla z) \\
        \tau z_t= \Delta z -z +u
    \end{cases}
\end{equation}
where $\tau \in \left \{ 0,1 \right \}$, and $f(u) = ru- \frac{\mu u^2}{\ln^p(u+e)}$, with $r \in \mathbb{R}$, $\mu \geq 0$, and $p\geq 0$, under the homogeneous Neumann boundary condition
\begin{equation} \label{bdry}
    \frac{\partial u}{\partial \nu} =\frac{\partial v}{\partial \nu} =0 ,\qquad (x,t) \in \partial \Omega \times (0, T_{\rm max}),
\end{equation}
where $T_{\rm max} \in (0,\infty]$ is the maximal existence time for classical solutions.
Chemotaxis system, which describes the movement of organisms in response to chemical stimuli has been intensively studied since the mathematical equations were first introduced in \cite{Keller}. The simplest chemotaxis model, known as Keller Segel system, in two spatial dimension has attracted many researchers in the field partly because of its intriguing phenomenon called critical mass. This principle states that if the mass of solutions is less than a certain critical number, called the critical mass, solutions remain global and bounded, whereas if the mass exceeds this critical threshold, solutions may blow-up in finite time. For more details about this topic, readers are referred to \cite{Dolbeault,Dolbeault1,Mizoguchi, NSY,OY} for some global existence results and  \cite{Nagai1,Nagai2,Nagai4,Nagai3} for blow-up results.\\

In the context of one-species chemotaxis systems, the presence of logistic sources, $ru-\mu u^2$, can avoid blow-up (see \cite{Winkler-logistic,Tello+Winkler}) in high dimension $n \geq 3$. In two dimensional bounded domain, the sub-logistic, $ru -\frac{\mu u^2}{\ln^p(u+e)}$ with $p \in [0,1]$, is sufficient to guarantee the global existence and boundedness of solutions (see \cite{Tian4} when $p\in (0,1)$ with $\tau =1$ and \cite{Tian2} $p=1$ with $\tau =0$). Moreover, the logistic sources is even adequate to prevent blow-up for degenerate chemotaxis system \cite{MW2022}. This was later improved in \cite{Minh2} by replacing the logistic sources by sub-logistic ones. \\

The effect of logistic sources or sub-logistic sources on blow-up prevention in two-species chemotaxis models is quite limited, especially when the logistic sources appear only in the first equation of the  system \eqref{KS2}. 
In \cite{THZ}, the authors study the global existence and long time behavior of solutions to the following system 
\begin{equation} \label{KS3}
    \begin{cases}
        u_t = \Delta u -\nabla \cdot (u \nabla v)+r_1u- \mu_1 u^2\\
      v_t = \Delta v -v+w,\\
        w_t =\Delta w -\nabla \cdot (w \nabla z) +r_2w -\mu_2 w^2 \\
        0= \Delta z -z +u,
    \end{cases}
\end{equation}
with $r_1,r_2 \in \mathbb{R}$ and $\mu_1,\mu_2>0.$ It was shown that if $\mu_1 \mu_2 $ is sufficiently large then all solutions to \eqref{KS3} are global and bounded for any $n \geq 1$. 
For further studies on global existence and equilibrium solutions for two species with logistic sources appearing in two cell density equations, readers can refer to \cite{XM, TW2014, CNT}. In fact, the analysis framework to prove the global existence of solutions to \eqref{KS3} for sufficiently large  $\mu_1\mu_2$ is similar to the one for one species \cite{Winkler-logistic}. However, there has no result so far considering the presence of a logistic source in only one species. Our purpose is to address that the appearance of the sub-logistic sources in one species can effectively eliminate the occurrence of finite time or infinite time blow-up solutions in two dimensional domains. Precisely, we have the following theorem:
\begin{theorem} \label{p}
    Assume that $f(u)=ru -\frac{\mu u^2}{\ln^p(u+e)}$, where $r,\mu >0$, and $ p\in [0, 1)$, and nonnegative initial data $u_0,w_0 \in C^0(\bar{\Omega})$. Then there exists a unique quadruple $(u,v,w,z)$ of nonnegative functions
    \begin{align*}
        u &\in C^0 \left ( \bar{\Omega}\times[0,\infty) \right ) \cap C^{2,1} \left ( \bar{\Omega}\times(0,\infty) \right ),  \\
        v &\in C^{2,0} \left ( \bar{\Omega}\times(0,\infty) \right ), \\
        w &\in C^0 \left ( \bar{\Omega}\times[0,\infty) \right ) \cap C^{2,1} \left ( \bar{\Omega}\times(0,\infty) \right ) \qquad \text{ and} \\
        z &\in C^{2,0} \left ( \bar{\Omega}\times(0,\infty) \right ),
    \end{align*}
    which solve the system \eqref{KS2} in the classical pointwise sense in $\Omega \times (0,\infty)$. Moreover, 
    \begin{equation} \label{gb}
        \sup_{t \in (0,\infty)} \left \{ \left \| u(\cdot,t) \right \|_{L^\infty(\Omega)}+\left \| v(\cdot,t) \right \|_{L^{\infty}(\Omega)}+\left \| w(\cdot,t) \right \|_{L^\infty(\Omega)}+\left \| z(\cdot,t) \right \|_{L^{\infty}(\Omega)} \right \} <\infty.
    \end{equation}
\end{theorem}
\begin{remark}
    Theorem \ref{p} works for both parabolic-elliptic systems when $\tau =0$ and fully parabolic systems when $\tau =1$.
\end{remark}
In the following sections, we will briefly recall the local well-posedness results for solutions to the system \eqref{KS2} in Section \ref{section2}, and explore the mechanisms behind blow-up prevention by sub-logistic sources in Section \ref{section3}.

\section{Local existence} \label{section2}
The local existence of solutions to the system \eqref{KS2} under homogeneous Neumann boundary conditions can be proved by adapting approaches that are well-established in the context chemotaxis models with logistic sources (see \cite{TW}[Theorem 2.1] for parabolic-elliptic models and \cite{Winkler-2011}[Lemma 1.1] for fully parabolic models). 
\begin{lemma} 
    Suppose that $u_0$ and $v_0$ are in $C^0(\bar(\Omega))$ are nonnegative. Then there exist $T_{\rm max} \in (0, \infty]$ and a unique quadruple $(u,v,w,z)$ of nonnegative functions from $C^0(\bar{\Omega}\times (0,T_{\rm max}))\cap C^{2,1}(\bar{\Omega}\times (0,T_{\rm max})) $ solving \eqref{KS2} under boundary condition \eqref{bdry} classically in $\Omega \times (0,T_{\rm max})$. Moreover, if $T_{\rm max }<\infty$, then
    \begin{equation} \label{lc}
       \limsup_{t\to T_{\rm max}} \left ( \left \| u(\cdot,t) \right \|_{L^\infty(\Omega)}+\left \| w(\cdot,t) \right \|_{L^\infty(\Omega)} \right) = \infty.
    \end{equation}
\end{lemma}

\section{Global boundedness with sub-logistic sources. Proof of Theorem \ref{p}} \label{section3}
The subsequent lemma holds a central position in this paper, serving as a cornerstone of our work. A noteworthy innovation introduced herein is the functional described by \eqref{keyfuctional}, with the specific values of the positive parameters $A$ and $B$ to be determined subsequently in the analysis.  It is notable that, within the context of inequality \eqref{prf.priori.12}, we identify a unique choice for $A$ that depends on the parameter $\epsilon$, resulting in the nonpositivity of the first term on the right-hand side of \eqref{prf.priori.12}.
\begin{lemma} \label{LlnL}
Under the assumptions as in Theorem \ref{p}, there exists a positive constant $C$ such that
    \begin{equation}
         \int_\Omega u(\cdot,t) \ln (u(\cdot,t)+e)+\int_\Omega w(\cdot,t) \ln (w(\cdot,t)+e) +\tau \int_\Omega |\nabla v(\cdot,t)|^2 +\tau \int_\Omega |\nabla z(\cdot,t)|^2  <C,
    \end{equation}
\end{lemma}
  for all $t\in (0,T_{\rm max})$.
\begin{proof}
    We define 
    \begin{align}\label{keyfuctional}
         y(t) := \int_\Omega u\ln(u+e)+ \int_\Omega w\ln(w+e) +\frac{A}{2} \int_\Omega |\nabla v|^2 +\frac{B}{2} \int_\Omega |\nabla z|^2 ,
    \end{align}
    where $A,B>0$ will be determined later. Differentiating $y$ in time, we obtain
   \begin{align} \label{prf.priori.1}
        y'(t) +y(t)&= \int_\Omega \left ( \ln(u+e)+\frac{u}{u+e} \right ) \left ( \Delta u- \nabla \cdot (u\nabla v) +ru -\frac{\mu u^2}{\ln^p(u+e)} \right ) \notag \\
        &\int_\Omega \left ( \ln(w+e)+\frac{w}{w+e} \right ) \left ( \Delta w- \nabla \cdot (w\nabla z) \right ) \notag \\
        &+ \tau A \int_\Omega \nabla v \cdot \nabla \left ( \Delta v -v +w \right ) \notag \\
        &+\tau B \int_\Omega \nabla z \cdot \nabla \left ( \Delta z -z +u \right ) \notag \\
        &:= I_1 +I_2+I_3+I_4.
    \end{align}
    In case $\tau =1$, we use integration by parts to obtain:
    \begin{align} \label{prf.priori.2}
        I_1 &= -\int_\Omega \frac{|\nabla u|^2}{u+e} -\int_\Omega \frac{e|\nabla u|^2}{(u+e)^2} +\int_\Omega \left (\frac{u}{u+e}+\frac{eu}{(u+e)^2}  \right )  \nabla u \cdot \nabla v  \notag \\
        &+ \int_\Omega\left ( \ln(u+e)+\frac{u}{u+e} \right ) \left ( ru -\frac{\mu u^2}{\ln^p(u+e)} \right ). 
    \end{align}  
    Let us define
    \[\phi(u):=\int_0 ^u \frac{s}{s+e}+\frac{es}{(s+e)^2}\, ds,\]
    we see that $\phi(u) \leq u$. When $\tau =0$, by integration by parts and elementary inequality, we make use of the second equations to obtain that:
    \begin{align}\label{prf.priori.3}
        \int_\Omega \left (\frac{u}{u+e}+\frac{eu}{(u+e)^2}  \right )  \nabla u \cdot \nabla v &= \int_\Omega \nabla \phi(u) \cdot \nabla v =-\int_\Omega \phi(u) \Delta v \notag \\
        &= \int_\Omega \phi(u)(w-v) \leq \int_\Omega uw \notag \\
        &\leq \epsilon \int_\Omega w^2 +\frac{1}{4\epsilon}\int_\Omega u^2 \notag \\
        &\leq \epsilon \int_\Omega w^2+\epsilon \int_\Omega u^2 \ln^{1-p}(u+e)+c,
    \end{align}
    where the last inequality comes from the fact that for any $\delta>0$, there exists $C>0$ depending on $\delta$ such that
\begin{align*}
     u^2 \leq \delta u^2\ln^{1-p}(u+e)+C(\delta),  \qquad 0\leq p <1.
\end{align*}
    We apply similar argument in case $\tau =1$ to obtain
    \begin{align} \label{prf.priori.3'}
        \int_\Omega \left (\frac{u}{u+e}+\frac{eu}{(u+e)^2}  \right )  \nabla u \cdot \nabla v &=-\int_\Omega \phi(u) \Delta v \notag \\
        &\leq \epsilon \int_\Omega (\Delta v)^2 +\frac{1}{4\epsilon} \int_\Omega \phi^2(u) \notag \\
        &\leq \epsilon \int_\Omega (\Delta v)^2 +\frac{1}{4\epsilon}  \int_\Omega u^2 \notag \\
        &\leq \epsilon \int_\Omega (\Delta v)^2+ \epsilon \int_\Omega u^2 \ln^{1-p}(u+e)+c ,
    \end{align}
One can verify that for any $\epsilon>0$, there exists $c>0$ depending on $\epsilon$ such that
    \begin{align} \label{prf.priori.4}
    \int_\Omega\left ( \ln(u+e)+\frac{u}{u+e} \right ) \left ( ru -\frac{\mu u^2}{\ln^p(u+e)} \right ) \leq (\epsilon -\mu)\int_\Omega u^2 \ln^{1-p}(u+e)+c.
\end{align}
From \eqref{prf.priori.2} to \eqref{prf.priori.4}, we obtain that
\begin{align}\label{prf.priori.5}
     I_1 \leq  (2\epsilon -\mu)\int_\Omega u^2 \ln^{1-p}(u+e) + \epsilon \int_\Omega w^2 +c, \quad \text{for }\tau =0,
\end{align}
and,
\begin{align}\label{prf.priori.5'}
    I_1 \leq  (2\epsilon -\mu)\int_\Omega u^2 \ln^{1-p}(u+e) + \epsilon \int_\Omega (\Delta v)^2 +c, \quad \text{for }\tau =1.
\end{align}
By similar arguments, one can also obtain that
\begin{align}
    I_2 \leq -\int_\Omega  \frac{|\nabla w|^2}{w+e} + \epsilon \int_\Omega w^2 +\epsilon \int_\Omega u^2 \ln^{1-p}(u+e) +c, \quad \text{for } \tau =0,
\end{align}
and
\begin{align}\label{prf.priori.6}
    I_2 \leq - \int_\Omega \frac{|\nabla w|^2}{w+e} +\epsilon \int_\Omega w^2 + \frac{1}{4\epsilon} \int_\Omega (\Delta z)^2, \quad \text{for }  \tau =1 . 
\end{align}
By integration by parts and elementary inequalities, we have
\begin{align}\label{prf.priori.7}
    I_3 &= -\tau A\int_\Omega (\Delta v)^2 -\tau A\int_\Omega |\nabla v|^2 -\tau A\int_\Omega w \Delta v \notag \\
    &\leq \tau (\frac{A^2}{4\epsilon}-A) \int_\Omega (\Delta v)^2 - \tau A\int_\Omega |\nabla v|^2 + \epsilon \tau  \int_\Omega w^2,
\end{align}
and
\begin{align}\label{prf.priori.8}
    I_4 &= - \tau B\int_\Omega (\Delta z)^2 -\tau B\int_\Omega |\nabla z|^2 - \tau B\int_\Omega u \Delta z \notag \\
    &\leq \tau(\epsilon -B) \int_\Omega (\Delta z)^2 -\tau B\int_\Omega |\nabla z|^2 + \tau \frac{B^2}{4\epsilon} \int_\Omega u^2 \notag \\
    &\leq \tau (\epsilon -B) \int_\Omega (\Delta z)^2 -\tau B\int_\Omega |\nabla z|^2 + \tau \epsilon \int_\Omega u^2 \ln^{1-p}(u+e)+c.    
\end{align}
One can verify that 
\begin{align}\label{prf.priori.9}
    \int_\Omega u \ln(u+e) + \int_\Omega w \ln(w+e) \leq \epsilon \int_\Omega u^2 \ln^{1-p}(u+e) + \epsilon \int_\Omega w^2 +c. 
\end{align}
From \eqref{prf.priori.5} to \eqref{prf.priori.9}, we have
\begin{align}\label{prf.priori.10}
    y'(t) +y(t) &\leq -\int_\Omega \frac{|\nabla w|^2}{w+e}+   (4\epsilon -\mu )\int_\Omega u^2\ln^{1-p}(u+e)+3\epsilon \int_\Omega w^2 \notag\\ &+\tau \left ( \frac{A^2}{4\epsilon} +\epsilon -A\right ) \int_\Omega (\Delta v)^2 
    + \tau \left ( \epsilon+\frac{1}{4\epsilon} -B \right ) \int_\Omega (\Delta z)^2 +c.
\end{align}
The third term can be controlled by Gagliardo–Nirenberg interpolation inequality
\begin{align}\label{prf.priori.11}
    3\epsilon \int_\Omega w^2 &\leq 3C_{GN}\epsilon \int_\Omega \frac{|\nabla w|^2}{w+e} \int_\Omega (w+e) +3C_{GN} \epsilon \left ( \int_\Omega (w+e) \right )^2 \notag \\
    & \leq 3C_{GN}\epsilon \left (\int_\Omega w_0+e|\Omega| \right) \int_\Omega \frac{|\nabla w|^2}{w+e} +3C_{GN}\epsilon \left (\int_\Omega w_0+e|\Omega| \right)^2.  
\end{align}
We substitute $\epsilon = \min \left \{ \frac{\mu}{4}, \frac{1}{3G_{GN}} \left (\int_\Omega w_0+e|\Omega|  \right)^{-1}  \right \}$ into \eqref{prf.priori.10} and \eqref{prf.priori.11} to obtain
\begin{align} \label{prf.priori.12}
    y'(t)+y(t) \leq \tau \left ( \frac{A^2}{4\epsilon} +\epsilon -A\right ) \int_\Omega (\Delta v)^2 
    + \tau \left ( \epsilon+\frac{1}{4\epsilon} -B \right ) \int_\Omega (\Delta z)^2 +c.
\end{align}
Now we can choose $A= 2\epsilon$ and $B = \epsilon + \frac{1}{4\epsilon}$ and make use of Gronwall inequality to complete the proof.
\end{proof}
We are now ready to prove the main theorem.

\begin{proof}[Proof of Theorem \ref{p}]
We employ the arguments from \cite{TW2015}[Lemma 4.2] with some modifications, and leverage  Lemma \ref{LlnL}  to derive the following inequality for all $t\in (0,T_{\rm max})$
\[
\left \|u (\cdot,t) \right \|_{L^2(\Omega)}+\left \|w (\cdot,t) \right \|_{L^2(\Omega)} <C.
\]
Subsequently, we apply Moser-type iterations, akin to \cite{TW2014}[Lemma 3.2] to establish the boundedness of  $u$ and $w$ in $\Omega \times (0,T_{\rm max})$. Combining this with \eqref{lc}, we conclude that $T_{\rm max} = \infty$. Employing elliptic regularity for $\tau =0$, and parabolic regularity for $\tau =1$, we ascertain the global boundedness of $v$ and $z$ in $\Omega \times (0,\infty)$. Consequently, we derive \eqref{gb}, thereby completing the proof.
\end{proof}

\printbibliography

\end{document}